\documentclass[11pt]{amsart}
\usepackage{graphicx}
\usepackage[active]{srcltx}
 \makeatletter
\renewcommand*\subjclass[2][2000]{%
  \def\@subjclass{#2}%
  \@ifundefined{subjclassname@#1}{%
    \ClassWarning{\@classname}{Unknown edition (#1) of Mathematics
      Subject Classification; using '1991'.}%
  }{%
    \@xp\let\@xp\subjclassname\csname subjclassname@#1\endcsname
  }%
}
 \makeatother
\usepackage{enumerate,url,amssymb,  mathrsfs}

\newtheorem{theorem}{Theorem}[section]
\newtheorem{lemma}[theorem]{Lemma}
\newtheorem*{lemma*}{Lemma}
\newtheorem{proposition}[theorem]{Proposition}
\newtheorem{corollary}[theorem]{Corollary}

\theoremstyle{definition}

\newtheorem{conjecture}[theorem]{Conjecture}

\theoremstyle{remark}

\numberwithin{equation}{section}

\def\XXint#1#2#3{{\setbox0=\hbox{$#1{#2#3}{\int}$}
\vcenter{\hbox{$#2#3$}}\kern-.5\wd0}}

\def\le{\leqslant}
\def\ge{\geqslant}
\begin{document}

\title[{Gradient of solution to  Poisson equation and related operators}]{Gradient of solution of the Poisson equation in the unit ball and related operators}



 \author{ David Kalaj  }
\address{ University of Montenegro, Faculty of Mathematics, Dzordza Va\v singtona  bb, 81000 Podgorica, Montenegro}
 \email{  davidk@ac.me}

 \author{ Djordjije Vujadinovi\'c }
\address{ University of Montenegro, Faculty of Mathematics, Dzordza Va\v singtona  bb, 81000 Podgorica, Montenegro}
 \email{  djordjijevuj@t-com.me}
 \date{}
\begin{abstract}
In this paper we determine the $L^1\to L^1$  and $L^{\infty}\to L^\infty$ norms of an integral operator $\mathcal{N}$ related to the gradient of the solution of Poisson equation in the unit ball with vanishing boundary data in sense of distributions.
 \end{abstract}

\maketitle

\section{Introduction and Notation}

We denote by $\mathbf{B}={B}^{n}$ and $\mathbf{S}={S}^{n-1}$ the unit ball and the unit sphere in $\mathbf{R}^n$ respectively.  We will assume that $n>2$ (the case $n=2$ has been already treated in \cite{Kalaj, Kalaj1}). By the vector norm $|\cdot|$ we  consider  the standard Euclidean distance $|x|=(\sum_{i=1}^{n}x_{i}^{2})^{\frac{1}{2}}.$

The norm of an operator $T: X\rightarrow Y$ defined on the normed  space $X$ with image in the normed space $Y$
is defined as
$$\|T\|=\sup\{\|Tx\|:\|x\|=1\}.$$

Let $G$ be the Green function, i.e., the function
$$G(x,y)=c_{n}\left(\frac{1}{|x-y|^{n-2}}-\frac{1}{[x,y]^{n-2}}\right),$$ where $$c_{n}=\frac{1}{(n-2)\omega_{n-1}},$$ where $\omega_{n-1}$ is the Hausdorff measure of $S^{n-1}$ and $$ [x,y]:=|x|y|-y/|y||=|y|x|-x/|x||.$$

The Poisson kernel $P$ is defined
$$P(x,\eta)=\frac{1-|x|^2}{|x-\eta|^{n}},\enspace |x|<1,\eta\in S^{n-1}.$$
We are going to consider the Poisson equation
\begin{equation}\label{eqpo}\left\{\begin{array}{rr}
                     \triangle u(x) &= g, x\in \Omega\\
                        &u|_{\partial \Omega}=f
                     \end{array}\right.\end{equation}
 where $f:S^{n-1}\rightarrow \mathbf{R}$ is bounded integrable function on the unit sphere $S^{n-1},$ and  $g:B^{n}\rightarrow \mathbf{R}$ is a continuous function.

The solution of the equation in the sense of distributions is given by
\begin{equation}
u(x)=P[f](x)-{\mathcal G}[g](x):=\int_{S^{n-1}}P(x,\eta)f(\eta)d\sigma(\eta)-\int_{B^{n}}G(x,y)g(y)dy,
\end{equation}
 $|x|<1.$
 Here $d\sigma$ is the normalized Lebesgue $n-1$ dimensional measure of the unit sphere $\mathbf{S}=S^{n-1}$.

Our main focus of observation is related to the special case of Poisson equation with Dirichlet boundary condition \begin{equation}\label{eqpo}\left\{\begin{array}{rr}
                     \triangle u(x) &= g, x\in \Omega\\
                        &u|_{\partial \Omega}=0
                     \end{array}\right.\end{equation} where $g\in L^{\infty}(B^{n}).$  The weak solution is then given by
 \begin{equation}
 u(x)=-{\mathcal G}[g](x)=-\int_{\mathbf{B}}G(x,y)g(y)dy,|x|<1.
 \end{equation}

 The problem of estimating the norm of the operator $\mathcal G$ in case of various $L^{p}-$ spaces was established by both authors in \cite{Kalaj 2}.

 Since $$\nabla_x G(x,y)=c_{n}(2-n)\left(\frac{x-y}{|x-y|^{n}}-\frac{|y|^2x-y}{[x,y]^n}\right),$$ this naturally induces the differential operator related to Poisson equation \begin{align}\label{3.3}\mathcal{D}[g](x):=\nabla u(x)=\frac{1}{\omega_{n-1}}\int_{\mathbf{B}}\left(\frac{x-y}{|x-y|^{n}}-\frac{|y|^2x-y}{[x,y]^n}\right) g(y) dy.\end{align}
 Related to the problem of estimating the norm of the operator $\mathcal{D}$, we are going to observe  the operator $\mathcal{N}:L^{\infty}(B^{n})\rightarrow L^{\infty}(B^{n})$ defined by
 \begin{align}\label{3.4}\mathcal{N}[f](x)=\frac{1}{\omega_{n-1}}\int_{\mathbf{B}}\left|\frac{x-y}{|x-y|^{n}}-\frac{|y|^2x-y}{[x,y]^n}\right| f(y) dy.\end{align}
  The main goal of our paper is related to estimating various norms of the integral operator and $\mathcal{N}$. Then we use those results to obtain some norm estimates of the operator  $\mathcal{D}.$  The compressive study of this problem for $n=2$ has been done by the first author in \cite{Kalaj, Kalaj1} and by Dostani\'c in \cite{dost,dost1}. For related results we refer to the papers \cite{ah,akl}.

   \subsection{Gauss hypergeometric function}
    Through the paper we will often use the properties of the hypergeometric functions. First of all, the hypergeometric function $F(a,b,c,t)={_2F_1} (a,b;c;t)$ is defined by the series expansion
  $$\sum_{n=0}^{\infty}\frac{(a)_{n}(b)_{n}}{n!(c)_n}t^{n},\enspace \mbox{for}\enspace|t|<1,$$
   and by the continuation elsewhere. Here $(a)_n$ denotes shifted factorial, i.e. $(a)_{n}=a(a+1)...(a+n-1)$ and $a$ is any real number.\\
   The following identity will be used in proving the main results of this paper:

\begin{equation}\label{formula}\int_0^{\pi} \frac{\sin^{\mu-1}t}{(1+r^2-2
r \cos t)^{\nu}}
dt=\mathrm{B}\left(\frac{\mu}{2},\frac{1}{2}\right)F\left(\nu,\nu
+\frac{1-\mu}{2};\frac{1+\mu}{2},r^2\right)\end{equation} (see,
e.g., Prudnikov, Brychkov and Marichev \cite[2.5.16(43)]{pbm}), where $\mathrm{B}$ is the beta function.

By using the Chebychev's inequality one can easily obtain the following inequality for Gamma function (see \cite{drag}).
\begin{proposition} Let $m$, $p$ and $k$ be real numbers with $m,p > 0$ and $p > k >- m$: If
\begin{equation}\label{(3.9)} k (p -m-k)\ge 0 (\le 0)\end{equation}
then we have
\begin{equation}\label{(3.10)} \Gamma(p)\Gamma(m) \ge (\le ) \Gamma(p - k)\Gamma (m + k). \end{equation}
 \end{proposition}

 \subsection{M\"obius transformations of the unit ball}
 The set of isometries of the hyperbolic unit ball $B^n$ is a  Kleinian
subgroup of all M\"obius transformations of the extended space
$\overline{\mathbf{R}}^n$ onto itself denoted by
$\mathbf{Conf}(\mathbf{B})=\mathbf{Isom}(\mathbf{B})$. We refer
to the Ahlfors' book \cite{Ahl} for detailed survey to this class of
important mappings.
 \begin{equation}\label{mebius}T_{x}y=\frac{(1-|x|^2)(y-x)-|y-x|^{2}x}{[x,y]^2},\end{equation} and \begin{equation}\label{mebnorm} |T_{x}y|=\left|\frac{x-y}{[x,y]}\right|\end{equation}

 \begin{equation}\label{deri}dy=\left(\frac{1-|x|^2}{[z,-x]^2}\right)^{n}dz.\end{equation}

\section{The $L^\infty$ norm of the operator $\mathcal N$}
In this section we are going to find the norm of the operator $\mathcal N,$ defined in \eqref{3.4}, and by using this we estimate the norm of operator $\mathcal{D}$.
\begin{theorem}
\label{thm}
Let ${\mathcal N}:L^{\infty}(\mathbf{B})\rightarrow L^{\infty}(\mathbf{B})$ be the operator defined in \eqref{3.4}. Then
 $$\|\mathcal N\|_{L^\infty\to L^\infty}=\frac{2n \pi^{n/2}}{(n+1)\Gamma(n/2)}.$$
\end{theorem}

\begin{proof}
At the beginning, let us note that
$$\|\mathcal N\|_{L^\infty\to L^\infty}=\sup_{x\in \mathbf{B}}\int_{\mathbf{B}} \left|\frac{x-y}{|x-y|^{n}}-\frac{|y|^2x-y}{[x,y]^n}\right|dy.$$

So, we need to find
$$\sup_{x\in \mathbf{B}}K(x),$$ where $$K(x)=\int_{\mathbf{B}} \left|\frac{x-y}{|x-y|^{n}}-\frac{|y|^2x-y}{[x,y]^n}\right|dy.$$
Now we are going to use the change of variables $y=T_{-x}z(T_{x}y=z),$ where $T_{-x}:\mathbf{B}\rightarrow \mathbf{B}$ is the M\"{o}bius transform defined by
$$T_{-x}(z)=\frac{(1-|x|^2)(y+x)+x|z+x|^2}{[z,-x]^2}.$$ Now we use the following relations $|T_{x}(y)|=\left|\frac{x-y}{[x,y]}\right|$ and
$$x-T_{-x}(z)=\frac{(1-|x|^2)(-x|z|^2-z)}{[z,-x]^2},\enspace \left|x-T_{-x}(z)\right|=|z|\frac{1-|x|^2}{[z,-x]}.$$
We have that
\begin{equation}
\begin{split}\label{ide}
&\left|\frac{x-y}{|x-y|^{n}}-\frac{|y|^2x-y}{[x,y]^n}\right|\\
&=\frac{1}{|x-y|^{n}}\left|(x-y)-(|y|^2x-y)\left|\frac{x-y}{[x,y]}\right|^n\right|
\\&=\frac{1}{|x-y|^{n}}\left|(x-y)-(|y|^2x-y)|z|^n\right|\\
&=\frac{1}{|x-T_{-x}z|^{n}}\left|(x-T_{-x}z)-(|T_{-x}z|^2x-T_{-x}z)|z|^{n}\right|\\
&=\frac{[z,-x]^{n}}{(1-|x|^2)^{n}|z|^{n}}\left|\frac{(1-|x|^2)(-x|z|^2-z)}{[z,-x]^2}+\frac{(1-|x|^2)(z+x)}{[z,-x]^2}|z|^{n}\right|\\
&=\frac{[z,-x]^{n}}{(1-|x|^2)^{n}|z|^{n}}\frac{(1-|x|^2)}{[z,-x]^{2}}|z|\left|(-x|z|-z/|z|)+|z|^{n-1}(z+x)\right|\\
&=\frac{[z,-x]^{(n-2)}}{(1-|x|^2)^{(n-1)}|z|^{n-1}}\left||z|^{n-1}(z+x)-(x|z|+z/|z|)\right|.
\end{split}
\end{equation}
According to the identity  \eqref{ide}, we have
\begin{equation}
\begin{split}
I=&\sup_{x\in B^n} K(x)\\
&=\sup_{x\in B^n}(1-|x|^2)\int_{0}^{1}dr\int_{\mathbf{S}}\frac{\left |rx(r^{n-2}-1)+\xi(r^{n}-1) \right|}{|rx+\xi|^{n+2}}d\xi.
\end{split}
\end{equation}
Further we have the following simple inequality $$ \left|rx(r^{n-2}-1)+\xi(r^{n}-1)\right|\le (1-r^{n-2})|rx+\xi| +r^{n-2}-r^{n},$$ and thus $$\frac{\left|rx(r^{n-2}-1)+\xi(r^{n}-1)\right|}{|rx+\xi|^{n+2}}\le \frac{(1-r^{n-2})|rx+\xi|}{|rx+\xi|^{n+2}}+
\frac{r^{n-2}-r^n}{|rx+\xi|^{n+2}}.$$
So $$I\le \max_{x\in B^n}(1-|x|^2)\int_{0}^{1}dr\int_{\mathbf{S}}\frac{\left((1-r^{n-2}) |rx+\xi|+r^{n-2}-r^n\right)}{|rx+\xi|^{n+2}}d\xi.$$
Then we have $$\int_S\frac{d\xi}{|rx+\xi|^a}=\frac{\omega_{n-1}}{\int_0^\pi \sin^{n-2} t dt}\int_0^\pi \frac{\sin^{n-2} t dt}{(1+r^2|x|^2+2 r|x| \cos t)^{a/2}}.$$
By \eqref{formula} we obtain \[\begin{split}\int_0^\pi &\frac{\sin^{n-2} t dt}{(1+r^2|x|^2+2 r|x| \cos t)^{a/2}}\\&=\frac{\sqrt{\pi } \Gamma\left(\frac{1}{2} (-1+n)\right)}{\Gamma\left(\frac{n}{2}\right)} { \, F\left(a/2,1-\frac{n}{2}+a/2,\frac{n}{2},r^2 x^2\right)}.\end{split}\]
In view of $$\frac{\omega_{n-1}}{\int_0^\pi \sin^{n-2} t dt}=\frac{\frac{2\pi^{n/2}}{\Gamma[n/2]}}{\frac{\sqrt{\pi} \Gamma[1/2 (-1 + n)]}{\Gamma[n/2]}}=\frac{2\pi^{n/2}}{\sqrt{\pi}\Gamma[1/2 (-1 + n)]},$$ we then infer
 $$\int_S\frac{d\xi}{|rx+\xi|^a}=2\frac{\pi^{n/2} }{\Gamma\left(\frac{n}{2}\right)}\, {F}\left(a/2,1-\frac{n}{2}+a/2,\frac{n}{2},r^2 x^2\right).$$
  Hence $$I\le C_n \sup_{x\in\mathbf{B}}J(x)$$ with \[\begin{split}J(x)&=(1-|x|^2)\int_{0}^{1}(1-r^{n-2}){F}\left(\frac{3}{2},\frac{1+n}{2},\frac{n}{2},r^2 x^2\right)dr\\&+(1-|x|^2)\int_{0}^{1}(r^{n-2}-r^n)\frac{n-(n-4) r^2 |x|^2}{n \left(1-r^2 |x|^2\right)^3}dr=\int_0^1 K_r(x) dr.\end{split}\] Here $$K_r(x)=\sum_{m=0}^\infty A_m(r) |x|^{2m}$$ where $A_0(r)=1-r^n$, $r=|x|$, and for $m\ge 1$
\[\begin{split}A_m(r)&=  \frac{r^{2m-4}}{{2n}}{r^n \left(1-r^2\right) \left(-2 m (-2+n+2 m)+2 (1+m) (n+2 m) r^2\right)}\\&+r^{2m-4}\left(r^2-r^n\right) \left(-2 m (-2+2 m+n)+(1+2 m) (-1+2 m+n) r^2\right) \\ &\quad\times \frac{ \Gamma[\frac{n}{2}]
\Gamma[\frac{1+2m}{2}] \Gamma[ \frac{n+2 m-1}{2}]}{2\sqrt{\pi } m! \Gamma\left(\frac{1+n}{2}\right) \Gamma\left(\frac{n}{2}+m\right)}.\end{split}\]
Thus $$I\le a_0+\sum_{m=1}^\infty a_m |x|^{2m}$$ where $$a_0=\frac{2n \pi^{n/2}}{(n+1)\Gamma(n/2)},$$ and
 \[\begin{split}a_m&= \frac{2 (-3+n) n+4 (-2+n) m}{n (-3+n+2 m) (-1+n+2 m) (1+n+2 m)}\\ & \quad  -\frac{(-2+n) (-3+n+4 m) \Gamma\left(\frac{n}{2}\right) \Gamma\left(-\frac{1}{2}+m\right) \Gamma\left(\frac{1}{2} (-3+n)+m\right)}{8 \sqrt{\pi } \Gamma\left(\frac{1+n}{2}\right) \Gamma(1+m) \Gamma\left(\frac{n}{2}+m\right)}.\end{split}\]
Then $a_m<0$ if and only if $$b_m:=\frac{2 ((-3+n) n+2 (-2+n) m) \sqrt{\pi } m! \Gamma\left(\frac{1+n}{2}\right) \Gamma\left(\frac{n}{2}+m\right)}{(-2+n) n (-3+n+4 m) \Gamma\left(\frac{n}{2}\right) \Gamma\left(-\frac{1}{2}+m\right) \Gamma\left(\frac{3+n}{2}+m\right)}<1.$$
Then by \eqref{(3.10)} we have $$\Gamma\left(-\frac{1}{2}+m\right) \Gamma\left(\frac{3+n}{2}+m\right)\ge \Gamma\left(m\right) \Gamma\left(\frac{2+n}{2}+m\right)$$ and so $$b_m\le c(m):=\frac{2 m ((n-3) n+2 (n-2) m) \sqrt{\pi } \Gamma\left(\frac{1+n}{2}\right)}{(-2+n) (n+2 m) (-3+n+4 m) \Gamma\left(1+\frac{n}{2}\right)}. $$ Since the last expression increases in $m$ because $$c'(m)=\frac{2 \left((n-3)^2 n^2+4 (n-3) (n-2) n m+4 (6+(n-3) n) m^2\right) \sqrt{\pi } \Gamma\left(\frac{1+n}{2}\right)}{(n-2) (n+2 m)^2 (n+4 m-3)^2 \Gamma\left(1+\frac{n}{2}\right)}\ge 0$$  we have $$b_m\le\lim_{m\to \infty} c(m)= \frac{\sqrt{\pi } \Gamma\left(\frac{1+n}{2}\right)}{2 \Gamma\left(1+\frac{n}{2}\right)}<1.$$ So $$\sup K(x)= K(0)=\frac{2n \pi^{n/2}}{(n+1)\Gamma(n/2)},$$ what we needed to prove.

\end{proof}

\begin{corollary}
\label{main}
Let ${\mathcal D}$ be the mapping defined in \eqref{3.3} and $v=\nabla u={\mathcal D}g,$ $g\in L^{\infty}(B^{n}).$ Then

 $$\|v\|_{\infty}\leq \frac{2n \pi^{n/2}}{(n+1)\Gamma(n/2)}\|g\|_{\infty}.$$
\end{corollary}

\begin{proof}
At the beginning, let us note that
$$\nabla_x G(x,y)=c_{n}(2-n)\left(\frac{x-y}{|x-y|^{n}}-\frac{|y|^2x-y}{[x,y]^n}\right).$$
For $x\in \mathbf{B}$ we have
\begin{equation}
\begin{split}
\|\nabla u(x)\|&=\sup_{|\xi|=1}\left|\left<\int_{\mathbf{B}}\nabla G(x,y)g(y)dy,\xi\right>\right|\\&=\sup_{|\xi|=1}\left|\int_{\mathbf{B}} \left<\nabla G(x,y),\xi\right>g(y)dy\right|\\
&=(n-2)c_{n}\sup_{|\xi|=1}\left|\int_{\mathbf{B}} \left<\frac{x-y}{|x-y|^{n}}-\frac{|y|^2x-y}{[x,y]^n},\xi\right>g(y)dy\right|\\
&\leq (n-2)c_{n}\int_{\mathbf{B}} \sup_{|\xi|=1}\left|\left<\frac{x-y}{|x-y|^{n}}-\frac{|y|^2x-y}{[x,y]^n},\xi\right>\right||g(y)|dy\\
&=  (n-2)c_{n}\int_{\mathbf{B}} \left|\frac{x-y}{|x-y|^{n}}-\frac{|y|^2x-y}{[x,y]^n}\right||g(y)|dy.\\
\end{split}
\end{equation}

So, we obtain the upper estimate for the gradiente of $u,$ i.e.
$$\|\nabla u\|_{\infty}\leq\|\mathcal{N}\|\|g\|_{\infty}.$$

\end{proof}

\section{The $L^1$ norm of the operator $\mathcal N$}
In the sequel let us state a well-known result related to the Riesz potential.

Let $\Omega$ be a domain of $R^{n},$ and let $|\Omega|$ be its volume. For $\mu \in (0, 1]$ define the
operator $V_{\mu}$ on the space $L^{1}(\Omega)$ by the Riesz potential
$$(V_{\mu}f)(x)=\int_{\Omega}|x-y|^{n(\mu-1)}f(y)dy.$$
The operator $V_{\mu}$ is defined for any $f \in L^{1}(\Omega),$ and $V_{\mu}$ is bounded on $L^{1}(\Omega),$ or
more generally we have the next lemma.
\begin{lemma}
\label{lema}
 (\cite{gt}, pp. 156-159]). Let $V_{\mu}$ be defined on the $L^{p}(\Omega)$ with $p > 0.$ Then
$V_{\mu}$ is continuous as a mapping $V_{\mu} : L^{p}(\Omega) \rightarrow L^{q}(\Omega),$ where $1 \leq q \leq\infty,$ and
$$0 \leq \delta = \delta(p, q) =\frac{1}{p}-\frac{1}{q}<\mu.$$

Moreover, for any $f \in L^{p}(\Omega)$
$$\|V_{\mu}f\|_{q} \leq \left(\frac{1}{\mu-\delta}\right)^{1-\delta}\left(\frac{\omega_{n-1}}{n}\right)^{1-\mu}|\Omega|^{\mu-\delta}\|f\|_{p}.$$
\end{lemma}
\begin{theorem}\label{l11}
The norm of the operator $\mathcal{N}:L^1\to L^1$ is $\frac{1}{n-2}$.
\end{theorem}
\begin{corollary}\label{l1}
Let $g\in L^1(\mathbf{B})$ and $v=\nabla u=\mathcal D[g]$. Then $$\|v\|_{1}\leq \frac{1}{n-2}\|g\|_{1}.$$
\end{corollary}
In order to prove  Theorem~\ref{l11} we need the following lemma

\begin{lemma}\label{needh}
Let $H(x,y)=\frac{y-x}{|y-x|^n}-\frac{|x|^2 y-x}{[y,x]^n}$, and let $$\mathcal{H}[g]=\frac{1}{(n-2)\omega_{n-1}}\int_{\mathbf{B}} |H(x,y)| g(y) dy.$$ Then
\begin{equation}\label{polu} \|\mathcal{H}\|_{L^\infty\to L^\infty}=\frac{1}{(n-2)\omega_{n-1}}\int_{\mathbf{B}} |H(0,y)| dy=\frac{1}{n-2}.\end{equation}
\end{lemma}

\begin{proof}[Proof of Lemma~\ref{needh}]

We need to find $$\sup_x \int_{\mathbf{B}} |H(x,y)|  dy.$$ We will show that its supremum is achieved for $x=0$.
We first have $$|H(x,y)|\le  K(x,y)+L(x,y)=|x-y|\left(\frac{1}{|x-y|^n}-\frac{1}{[x,y]^n}\right)+\frac{|y|(1-|x|^2)}{[x,y]^n}.$$

Further  we have
\begin{equation}
\sup_{x\in B^{n}}\int_{B^{n}}\left|K(x,y)\right|dy=\sup_{x\in B^{n}}\int_{B^{n}}\frac{1}{|x-y|^{n-1}}\left|1-\left|\frac{x-y}{[x,y]}\right|^{n}\right|dy.
\end{equation}
We use the change of variables $z=T_{x}y,\enspace\mbox{i.e.}\enspace T_{-x}z=y,$ where $T_{x}y$ is the M\"{o}bius transform
$$T_{x}y=\frac{(1-|x|^2)(y-x)-|y-x|^{2}x}{[x,y]^2},\enspace |T_{x}y|=\left|\frac{x-y}{[x,y]}\right|.$$ We obtain
$$dy=\left(\frac{1-|x|^2}{[z,-x]^2}\right)^{n}dz.$$
Assume without loss of generality that $x=|x|e_1$. Further for $\xi=(\xi_1,...,\xi_n),$
\begin{equation}
\begin{split}
&\sup_{x\in \mathbf{B}}\int_{\mathbf{B}}\left|H(x,y)\right|dy=\sup_{x\in \mathbf{B}}\int_{\mathbf{B}}\frac{1}{|x-T_{-x}z|^{n-1}}|1-|z|^{n}|\frac{(1-|x|^2)^n}{[z,-x]^{2n}}dz\\
&=\sup_{x\in \mathbf{B}}(1-|x|^2)^{n}\int_{\mathbf{B}}\frac{(1-|z|^{n-2})}{\left|\frac{x[z,-x]^{2}-(1-|x|^2)(x+z)-|x+z|^{2}x}{[z,-x]^2}\right|^{n-1}}\frac{dz}{[z,-x]^{2n}}\\
&=\sup_{x\in \mathbf{B}}(1-|x|^2)^{n}\int_{\mathbf{B}}\frac{(1-|z|^n)}{|z|^{n-1}\left|\frac{1-|x|^2}{[z,-x]}\right|^{n-1}}\frac{dz}{[z,-x]^{2n}}\\
&=\sup_{x\in \mathbf{B}}(1-|x|^2)^{n-(n-1)}\int_{\mathbf{B}}\left(\frac{1-|z|^n}{|z|^{n-1}}\right)[z,-x]^{(n-1)-2n}dz\\
&=\sup_{x\in \mathbf{B}}(1-|x|^2)\int_{0}^{1}(1-r^{n})r^{n-(n-1)-1}dr\int_{\mathbf{S}}\frac{d\xi}{|rx+\xi|^{2n-(n-1)}}\\
&=\sup_{x\in \mathbf{B}}(1-|x|^2)\int_{0}^{1}(1-r^n)dr \int_{\mathbf{S}} \frac{d\xi}{(r^{2}|x|^{2}+2r|x|\xi_{1}+1)^{\frac{n+1}{2}}}\\
&=\frac{2 \pi ^{\frac{1}{2} (-1+n)}}{\Gamma\left[\frac{1}{2} (-1+n)\right]}\sup_{x\in \mathbf{B}}(1-|x|^2)\int_{0}^{1}(1-r^{n})dr
\int_{0}^{\pi}\frac{\sin^{n-2} t}{(r^{2}|x|^{2}+2r|x|\cos t+1)^{\frac{n+1}{2}}}dt,
\\&=\frac{2 \pi ^{n/2}}{\Gamma\left[\frac{n}{2}\right]}  \sup_{x\in \mathbf{B}}(1-|x|^2)\int_{0}^{1}(1-r^{n}) F\left(\frac{3}{2},\frac{1+n}{2},\frac{n}{2},r^2 |x|^2\right)dr\\&=\frac{2 \pi ^{n/2}}{\Gamma\left[\frac{n}{2}\right]}  \sup_{x\in \mathbf{B}} J(x)= \omega_{n-1} \sup_{x\in \mathbf{B}} J(x)
\end{split}
\end{equation}
where \[\begin{split}&J(x)=\frac{n}{1 + n}-\sum_{m=1}^\infty e_m |x|^{2m}\le J(0),\end{split}\] with $$e_m=\frac{n \left(8 m^2+(n-1)^2+2 m (3n-5)\right) \Gamma\left(m-\frac{1}{2}\right) \Gamma\left(\frac{3}{2}+m+\frac{n}{2}\right) \Gamma\left(\frac{n}{2}\right)}{(2 m+n-1)^2 (2 m+n+1)^2 \sqrt{\pi } \Gamma[1+m] \Gamma\left(m+\frac{n}{2}\right) \Gamma\left(\frac{1+n}{2}\right)}.$$
In the first appearance of hypergemetric function we used \eqref{formula}.

On the other hand similarly we prove that \[\begin{split}L(x)&=\int_{\mathbf{B}}L(x,y)dy\\&=C_n'\left(1-|x|^2\right) {F}\left(1,\frac{1+n}{2},\frac{3+n}{2},|x|^2\right)\\&=C_n' (1-\sum_{m=1}^\infty \frac{2 (1+n)}{-1+4 m^2+4 m n+n^2}|x|^{2m})\\&\le L(0),\end{split}\] where $$C_n'=L(0)=\int_{\mathbf{B}}|y| dy= \frac{n}{n+1}\frac{\pi^{n/2}}{\Gamma[1+n/2]}=\frac{\omega_{n-1}}{n+1}.$$
Hence \[\begin{split}\sup_x \int_{\mathbf{B}} |H(x,y)| dy&\le \int_{\mathbf{B}} |H(0,y)| dy=\omega_{n-1}J(0)+L(0)\\&=\left(\frac{n}{n+1}+\frac{1}{n+1}\right)\omega_{n-1}=\omega_{n-1}.\end{split}\] This implies \eqref{polu}.
\end{proof}
\begin{proof}[Proof of Theorem~\ref{l1} and Corollary~\ref{l11}]
Since
$$\|\mathcal N\|_{L^{1}\rightarrow L^{1}}=\|\mathcal N^{\ast}\|_{L^{\infty}\rightarrow L^{\infty}},$$
where $\mathcal N^{\ast}$ is appropriate adjoint  operator and $$\mathcal N^{\ast}f(x)=\int_{\mathbf{B}}\overline{\mathcal N(y,x)}f(y)dy=\int_{\mathbf{B}}|H(x,y)|f(y)dy,\quad f\in L^{\infty}(B),$$ we have $$\|\mathcal N^{\ast}\|_{L^{1}\rightarrow L^{1}}=\|\mathcal H\|_{L^{\infty}\rightarrow L^{\infty}}.$$ So Theorem~\ref{l1} follows from Lemma~\ref{needh}.

On the other hand,  Corollary~\ref{l11} follows from the following inequality
$$ \|\mathcal D[g]\|_{L^{1}\rightarrow L^{1}}\leq \|\mathcal N\|_{L^{1}\rightarrow L^{1}}.$$
\end{proof}
At this point let us point out  the fact that $$\mathcal D:L^{p}(\mathbf{B},\mathbf{R})\rightarrow L^{p}(\mathbf{B},\mathbf{R}^{n}),$$ where
$L^{p}(\mathbf{B},\mathbf{R}^{n}) $ is the appropriate Lebesgue space of vector-functions. By $\|\mathcal D\|_{p}$ we denote the norm of the operator $\mathcal D.$

By using the Ries-Thorin interpolation theorem we obtain the next estimates of the norm for the operators $\mathcal N$ and $\mathcal{D}$
\begin{corollary}
Let us denote by $\|\mathcal N\|_i:=\|{\mathcal N}\|_{L^{i}\rightarrow L^{i}}$, $i\in\{1,\infty\}$. Then
$$\|\mathcal D\|_p<\|\mathcal{N}\|_p\leq \|{\mathcal N}\|_{1}^{\frac{1}{p}}\|{\mathcal N}\|_{\infty}^{\frac{p-1}{p}},1<p<\infty.$$
\end{corollary}

\begin{conjecture}
We know that ${\mathcal D},\mathcal{N}$ maps $ L^{p}(\mathbf{B})$ into $L^{\infty}(\mathbf{B})$ for $p>n$. We have that
$$\|{\mathcal N} g \|_{\infty}\leq  A_{p}\|g\|_{p},$$ and $$\|{\mathcal D} g \|_{\infty}\leq  B_{p}\|g\|_{p},$$  where $$A_p=\frac{1}{\omega_{n-1}}\sup_{x\in\mathbf{B}}\left(\int_{\mathbf{B}}\left|\frac{x-y}{|x-y|^{n}}-\frac{|y|^2x-y}{[x,y]^n}\right|^{q}dy\right)^{1/q},$$ and

$$B_p=\frac{1}{\omega_{n-1}}\sup_{x\in\mathbf{B},|\eta|=1}\left(\int_{\mathbf{B}}\left|\left<\frac{x-y}{|x-y|^{n}}-\frac{|y|^2x-y}{[x,y]^n},\eta\right>\right|^{q}
dy\right)^{1/q}.$$
  Then  we conjecture that $$A_p=\frac{1}{\omega_{n-1}}\left(\int_{\mathbf{B}}\left(\frac{1}{|y|^{n-1}}-|y|\right)^{q}dy\right)^{1/q}=\omega_{n-1}^{-\frac{1}{p}}\left(\frac{\Gamma[1+q] \Gamma\left[1+\left(-1+\frac{1}{n}\right) q\right]}{n \Gamma\left[2+\frac{q}{n}\right]}\right)^{1/q},$$ and

  \[\begin{split}B_p&=\frac{1}{\omega_{n-1}}\sup_{|\eta|=1}\left(\int_{\mathbf{B}}|\left<y,\eta\right>|^q\left(\frac{1}{|y|^{n}}-1\right)^{q}dy\right)^{1/q}\\&
  =\omega_{n-1}^{-\frac{1}{p}}\left(\frac{\Gamma\left[\frac{n}{2}\right] \Gamma[1+q] \Gamma\left[\frac{1}{2} (-1+n+q)\right] \Gamma\left[1+\left(-1+\frac{1}{n}\right) q\right]}{n \Gamma\left[\frac{1}{2} (-1+n)\right] \Gamma\left[\frac{n+q}{2}\right] \Gamma\left[2+\frac{q}{n}\right]}\right)^{1/q}.\end{split}\]
\end{conjecture}

\end{document}